\documentclass[12pt]{article}
\usepackage{amsfonts,latexsym}
\usepackage{amsmath}
\usepackage{amsfonts}
\usepackage{amssymb}

\usepackage{amssymb}
\usepackage{amsmath,amsthm}
\usepackage{enumerate}
\usepackage{framed}
\usepackage{verbatim}
\usepackage[algo2e]{algorithm2e}
\usepackage{algorithm}
\usepackage{algorithmic}
\usepackage{color}
\newtheorem{thm}{Theorem}[section]

\newcommand*{\QEDB}{\hfill\ensuremath{\square}}%
\theoremstyle{definition}
\newtheorem{defn}[thm]{Definition}
\newtheorem{lem}[thm]{Lemma}

\newtheorem{cor}[thm]{Corollary}

\newtheorem{rem}[thm]{Remark}

\theoremstyle{remark}

\newtheorem*{notn}{Notation}

\usepackage{mathtools}
\mathtoolsset{showonlyrefs}
\renewcommand{\P}{\mathbb{P}}

\title{Connectivity of the k-out Hypercube}
\author{ Michael Anastos\footnote{email:manastos@andrew.cmu.edu}  \vspace{5mm}
\\Department of Mathematical Sciences,
\\Carnegie Mellon University,
\\Pittsburgh PA 15213.
\date{}}

\begin{document}
\maketitle
\nocite{*}

\begin{abstract}
In this paper we study  the connectivity properties of the random subgraph of the $n$-cube generated by the $k$-out model and denoted by $Q^n(k)$. Let $k$ be an integer, $1\leq k \leq n-1$. We let $Q^n(k)$ be the graph that is generated by independently including for every $v\in V(Q^n)$ a set of $k$ distinct edges chosen uniformly from all the $\binom{n}{k}$ sets of distinct edges that are incident to $v$.  
We study connectivity the properties of $Q^n(k)$ as $k$ varies. 
We show that w.h.p.\@\footnote{we say that a sequence of events $\{\mathcal{E}_n\}$  holds with high probability (w.h.p.\@) or equivalency almost surely if $\P(\mathcal{E}_n)\to  1 $ as $n\to \infty$. } $Q^n(1)$ does not contain a giant component i.e.\@ a component that  spans $\Omega(2^n)$ vertices. Thereafter we show that such a component emerges when $k=2$. In addition the giant component spans all but $o(2^n)$ vertices  and hence it is unique. We then establish the connectivity threshold found at $k_0= \log_2 n -2\log_2\log_2 n $. The threshold is sharp in the sense that  $Q^n(\lfloor k_0\rfloor )$ is disconnected but $Q^n(\lceil k_0\rceil+1)$ is connected  w.h.p. Furthermore we show that   w.h.p.\@ $Q^n(k)$ is $k$-connected
 for every $k\geq \lceil k_0\rceil+1$.
\end{abstract}

\section{Introduction}

The n-dimensional cube, denoted by $Q^n$, is the graph with vertex set  $V=\{0,1\}^n$ in which two vertices are connected if and only if  they differ into precisely one coordinate. Clearly $Q^n$ is an $n$-regular bipartite graph on $2^n$ vertices. In this paper we study the random subgraph of the $n$-cube generated by the $k$-out model and denoted by $Q^n(k)$.
Let $k$ be an integer, $1\leq k \leq n-1$. We let $Q^n(k)$ be the graph that is generated by independently including for every $v\in V(Q^n)$ a set of $k$ distinct edges chosen uniformly from all the $\binom{n}{k}$ sets of distinct edges that are incident to $v$.  

Random subgraphs of $Q^n$ can be generated in various ways. 
The most usual way to generate such  graphs is either using the $G(Q^n,p)$ model or the $(Q^n)_t$ random process. In the $G(Q^n,p)$ model every edge of $Q^n$ is included independently with probability $0<p<1$. On the other hand the random process $(Q^n)_t$ is generated by starting with $(Q^n)_0$, the empty graph on $V$, and
extending $(Q^n)_i$ to $(Q^n)_{i+1}$ by adding to $(Q^n)_i$ an edge from $Q^n$,
 that is not currently present, uniformly at random. Various results on connectivity have been establish in both models. Burtin \cite{Burtin} was the first to study the connectivity of $G(Q^n,p)$. He proved that $G(Q^n,p)$ connected w.h.p.\@ when $p>\frac{1}{2}$ and is disconnected when $p<\frac{1}{2}$. His result was sharped by Erd\"os and Spencer who also conjectured that if $p\geq \frac{1+\epsilon}{n}$, $\epsilon>0$, then $G(Q^n,p)$  almost surely has a giant component. Their conjecture was  verified by Ajtai, 
 Koml\'os and Szemer\'edi \cite{Ajtai1982}. The connectivity of the random process $(Q^n)_t$ was studied by Bollob\'as, Kohayakawa and {\L}uczak \cite{Bollobas2},\cite{Bollobas1}. They established the following result. Let $\ell=O(1)$, and $\tau_\ell=\min\{t\in[n2^n]: \delta_t\geq \ell\}$. Then w.h.p\@ $(Q^n)_{\tau_\ell}$ is $\ell$-connected. Here by $\delta_t$ we denote the minimum degree of $(Q^n)_t$. Furthermore we say that a graph is  $\ell$-connected if it has more than $\ell$ vertices and remains connected whenever fewer than $\ell$ vertices are removed.
 
Observe that $\ell$-connectivity requires that the minimum degree of a graph is at least $\ell$. In both of the above models one has to wait until the minimum degree is $\ell$. Once this requirement is fulfilled  then the graph is $\ell$-connected w.h.p. One is therefore interested in models of a random graph which guarantee a certain minimum degree while not having too many edges. The $k$-out model meets this requirement. 

There have already been studies on connectivity properties of random graphs that are generated by the k-out model. For an arbitrary graph $G$ let $G(k)$ denote the random subgraph of $G$ that is generated by the $k$-out model, $1\leq k\leq \delta(G)$ (here $\delta(G)$ denotes the minimum degree of $G$). In the case that $G$ is the compete graph on $n$ vertices $K_n$ the following are known to hold w.h.p.\@ (see \cite{Alan}). First $K_n(1)$ is 
disconnected. Then $K_n(2)$ is connected, a proof of which can been found in the Scottish book \cite{BLMS:BLMS0265}. Furthermore Fenner and Frieze \cite{Fenner1982OnTC} show that for $k\geq 2$ we have that $K_n(k)$ is $k$-connected. The last theorem has been recently generalized by Frieze and Johansson \cite{Johansson} in the case where $k=O(1)$. They showed that for an arbitrary graph $G$ of minimum degree $\delta(G)\geq \big( \frac{1}{2}+\epsilon\big)n$ we have that the random graph $G(k)$ is $k$-connected for $2\leq k =O(1)$.     

In this paper we study connectivity properties of $Q^n(k)$  as $n\to\infty$. As we vary $k$ we ask whether some specific connectivity properties hold. Our results are summarized in the three theorems given below.  

\begin{thm}\label{first}
W.h.p.\@ $Q^n(1)$ does not contain a component spanning $\Omega(2^n)$ vertices.
\end{thm}
\begin{thm}\label{second}
W.h.p.\@ $Q^n(2)$ contains a unique giant component that spans all but $o(2^n)$ vertices.
\end{thm}
\begin{thm}\label{third}
Let $k_0= \log_2 n -2\log_2\log_2 n $. Then w.h.p.\@ the following hold,
\begin{enumerate}
    \item if $k\leq \lfloor k_0 \rfloor $ then $Q^n(k)$ is disconnected,
    \item if $k\geq \lceil k_0 \rceil+1$ then $Q^n(k)$ is $k$-connected.
\end{enumerate}
\end{thm}
The most surprisingly  feature of our results, as opposed to what someone might expect based on the results concerning $K_n(2)$, is that $Q^n(2)$ is not connected. Furthermore even though $Q^n(2)$ contains a giant component that spans all but $o(2^n)$ vertices we have that as $k$ increases $Q^n(k)$ persists on not being connected until $k$  passes $k_0$.  On the other hand, as is proved for $K_n(k)$, we are able to prove that if $Q^n(k)$ is connected then it is $k$-connected. 
\vspace{3mm}
\\The rest of the paper is split as follows.  In section 2 we give some notation and  preliminary results. Thereafter  in Sections 3 and 4 we give the proofs of Theorem \ref{first} and \ref{second} respectively. We continue by proving the first part of Theorem \ref{third}  in Section 5. We  give the proof of the second part  in section 6. We close with section 7.

\section{Notation-Preliminaries}
In this section we give some definitions and basic results that are
used throughout the paper. We use $V$ and $E$ in order to denote $V(Q^n)$ and $E(Q^n)$ respectively. 
\begin{defn}
We say that a graph $G$ is  $\ell$-connected if it has more than $\ell$ vertices and remains connected whenever fewer than $\ell$ vertices are removed.
\end{defn}
\begin{notn}
For $v\in V$ and $A\subset V$  we set $E(v,A):=\{vw \in E:  w\in A\}$.
Furthermore  for $A,B \subset V$ we set $E(A,B):= \{vw\in E(G): v\in A, w\in B\}$.
Finally we denote the quantities $|E(v,A)|$ and $|E(A,B)|$ by $d(v,A)$ and $d(A,B)$ respectively. 
\end{notn}
In various places we are going use the following inequalities (see \cite{Ch0}, \cite{Mc}). By $Bin(n,p)$ we  mean the $Binomial(n,p)$ random variable.
\begin{lem}\label{con1}
(\emph{Chernoff's bounds})  Let $X$ be distributed as a $Bin(n,p)$ random variable. Then for any 
$0 \leq \epsilon  \leq 1$,
$$\mathbb{P}\big(X\geq (1+\epsilon)np\big)\leq e^{-\frac{\epsilon^2np}{3}} $$
\end{lem}
\begin{lem}\label{con2}
(\emph{McDiarmid's Inequality}) Let $X_1,X_2,...,X_n$ be independent random variables with $X_i$ taking values in $A_i$. Further, let 
$f:\prod_{i\in[n]} A_i \mapsto \mathbb{R}$ and assume there exist $c_1,...,c_n\in \mathbb{R}$
such that whenever $x,x'$ differ only in their $i$-th coordinate we have
\begin{align}\label{lip}
|f(x)-f(x')|\leq c_i.
\end{align}
Then $\forall \epsilon>0$,
$$\mathbb{P}\big[f-\mathbb{E}[f(x)]\geq \epsilon \big]\leq \exp\bigg\{ \frac{-2\epsilon^2}{{\sum}_{i=1}^{n}c_i^2} \bigg\} $$
\end{lem}
\begin{rem}
In the setup above if $f$ satisfies  condition \eqref{lip} then so does $-f$. Therefore by applying McDiarmid's Inequality twice, once with $f$ and once with $-f$, we get that $\forall \beta >0$
$$\mathbb{P}\bigg\{f\notin \big[(1-\beta)\mathbb{E}\big(f(x)\big), (1+\beta)\mathbb{E}\big(f(x)\big)\big] \bigg\}\leq 2\exp\bigg\{ \frac{-2\big[\beta \mathbb{E}\big(f(x)\big)\big]^2}{{\sum}_{i}^{n}c_i^2} \bigg\} $$
\end{rem}
We will also use the following isoperimetric inequality,  see for example  Bollob\'as and Leader {\cite{Leader}}.
\begin{lem}\label{iso}
Let $A \subset V$. Then,
$$d(A,A)\leq \frac{|A|\log_2|A|}{2}.$$
Equivalently, since $d(A,V{ \setminus} A)=n|A|-2d(A,A)$,
we have
$$d(A,V{ \setminus} {A})\geq n|A|-|A|\log_2|A|.$$
\end{lem}
\begin{rem}\label{iso2}
The function $f(x)=nx-x\log_2x$ has a unique stationary point which is a maximum. Therefore  $\forall a\leq b\in [0,2^n]$ if $A\subset V$ and $|A|\in [a,b]$ then Lemma \ref{iso} implies that $d(A,V{ \setminus} A)\geq \min\{ f(a), f(b)\}.$
\end{rem}
Finally we are also going to use the following two results. For a proof of Lemma \ref{trees1} in the case where  the underline graph has maximum degree $D$ see Knuth \cite{Knuth}.
\begin{lem}\label{trees1}
For $v\in V$ there are at most $\binom{sn}{s} / [(n-1)s+1]$ trees $T$ such that $|T|=s$ and $v\in V(T)$.
\end{lem}
\begin{cor}\label{trees}
For $v\in V$ there are at most $(en)^{s}$ sets $S$ such that (i) $v\in 
S$, (ii) $|S|=s$ and (iii) $G[S]$
is connected.
\end{cor}
\begin{proof}
It follows directly from Lemma \ref{trees} and the following inequality,
\begin{align*}
&\binom{sn}{s} / [(n-1)s+1] \leq \binom{sn}{s}\leq\bigg(\frac{ens}{s}\bigg)^s=(en)^s.
\qedhere
\end{align*}
\end{proof}

\section{Structural properties of $Q^n(1)$}
We split this section into two parts.  In the first part we prove Theorem \ref{first}. Thereafter we 
split $[0,2^n]$ into sub-intervals and for each interval we
study the number of  components in $Q^n(1)$ with size in that interval. We use this information to prove Theorem \ref{second} in the next section.
\vspace{3mm}
\\We generate  $Q^n(1)$ in the following manner.
Every vertex $v\in V$ independently chooses a vertex $f(v)$ from those adjacent to it in $Q^n$ uniformly at random. Let $E_D$ be the set of arcs $\{(v,f(v)):v\in V\}$.   We set $G_D:=(V,E_D)$. Lastly we set $Q^n(1)$ be the simple graph that we get from $G_D$ when we  ignore orientation. 
\begin{rem}
$G_D$ is the union of in-arborescences and directed cycles. Moreover the in-arborescences can be chosen such that the root of every in-arborescence lies on a cycle.
\end{rem}
\subsection{The lack of a giant component}
\begin{lem}\label{small paths}
W.h.p.\@ $\nexists v,w\in V$ such that in $G_D$ there is a directed path from $u$ to $w$  of length larger than $n^2$.
\end{lem}
\begin{proof}
Let $v\in V$. Explore, by sequentially revealing the out-edges, the vertices that $v$ can reach (i.e\@ the vertices $v,f(v),f^2(v),...$). Suppose that  we have revealed the arcs $\big(v,f(v)\big),...,\big(f^{i-1}(v),f^i(v)\big)$ and that these arcs do not span a cycle. Then $\big(f^i(v),f^{i+1}(v)\big)$ is still distributed uniformly at random  over the $n$  arcs out of $f^{i}(v)$. Thus with probability $\frac{1}{n}$ we have $f^{i-1}(v)=f^{i+1}(v)$, in which case we say that $f^i(v)$ ``closes the path''. Let $B(v,n^2)$ be the event that
there exists a directed path out of $v$ of size larger than $n^2$ and    $A(v,i)$ be the event that 
$ \big(v,f(v)\big),...,\big(f^{i-1}(v),f^i(v)\big)$ do not span a cycle. Then,
\begin{align}
    \mathbb{P}\big (\exists v\in V: B(v,n^2)\big) &
    \leq\sum_{v\in V}\prod_{i\in[n^2]} \mathbb{P}\big( f^i(v) \text{ does not close the path}\vert A(v,i) \big)\\
    &\leq\sum_{v\in V} \prod_{i\in[n^2]} \bigg(1-\frac{1}{n}\bigg)=2^n\bigg(1-\frac{1}{n}\bigg)^{n^2}
    \leq 2^n \cdot e^{-\frac{1}{n}\cdot n^2}=o(1).
    \qedhere
\end{align}
\end{proof}
\vspace{3mm}
 \emph{\bf{Proof of Theorem \ref{first}}}. Let $Z$ be that number of unordered pairs $u,v\in V$ such that there is a path from $u$ to $v$ in $Q^n(1)$ (i.e.\@ $u,v$ belong to the same component in $Q^n(1)$).
Let $u,v$ be such a pair. Then in $G_D$ there exists a unique $w\in V$ such that there exists  di-paths (directed paths) from both $u$ and $v$ to $w$ that share no vertices other than $w$.  (here we use the convection that for every $v\in V$ there is a di-path from $v$ to $v$ of length 0). Set $Z=Z_S+Z_L$ where $Z_S$ counts the pairs of vertices where both of the corresponding paths have length at most $n^2$ and $Z_L$ counts the rest of the pairs. Lemma \ref{small paths} implies that there do not exist any di-paths of size larger than $n^2$ thus $Z_L=0$  w.h.p. 
\vspace{3mm}
\\In order to bound $Z_S$ for $v,w\in V$ let $\mathcal{P}_{u\mapsto w}$ be the set of all di-paths from $v$ to $w$ of length at most $n^2$. Furthermore for $u,v,w\in V$ let 
$\mathcal{P}_{u,v\mapsto w}\subset \mathcal{P}_{u\mapsto w}\times \mathcal{P}_{v\mapsto w}$ be the set of  all the pairs of di-paths $(P_1,P_2)\in    \mathcal{P}_{u\mapsto w}\times \mathcal{P}_{v\mapsto w}$ where $P_1$ and $P_2$ do not share any vertices other than $w$. We denote by $\mathbb{I}(\cdot)$  the indicator function. Therefore for a path $P$ we have  that $\mathbb{I}(P)=1$ in the event that $E(P)\subset E(G_D)$ and 0 otherwise. In addition for a set of paths $\mathcal{P}$ we have that 
$\mathbb{I}(\mathcal{P})=1$ in the event that there exists some $P\in \mathcal{P}$ such that  $\mathbb{I}(P)=1$ and 0 otherwise. Finally for a pair of paths $P_1,P_2$ we set $\mathbb{I}(P_1\wedge P_2)=\mathbb{I}(P_1)\mathbb{I}(P_2)$. Thus
\begin{align}\label{pathcount}
    \mathbb{E}(Z_S)&= \mathbb{E}\bigg[\sum_{u,v,w\in V}
    \sum_{(P_1,P_2)\in \mathcal{P}_{u,v\mapsto w}}
    \mathbb{I}
    \bigg( P_1 \wedge P_2\bigg) \bigg]=
  \sum_{u,v,w\in V} \sum_{(P_1,P_2)\in \mathcal{P}_{u,v\mapsto w}}   \mathbb{E}\big[\mathbb{I}\big(P_1 )\big] \mathbb{E}\big[\mathbb{I}\big( P_2)\big] \\ 
  &\leq  \sum_{u,v,w\in V} \sum_{(P_1,P_2)\in \mathcal{P}_{u\mapsto w}\times \mathcal{P}_{v\mapsto w}}   \mathbb{E}\big[\mathbb{I}\big(P_1 )\big] \mathbb{E}\big[\mathbb{I}\big( P_2)\big]
   = \sum_{u,v,w\in V}  \mathbb{E}\big[\mathbb{I}(\mathcal{P}_{u\mapsto w})
    \big]
    \mathbb{E}\big[\mathbb{I}(\mathcal{P}_{v\mapsto w}) \big] \\& 
    =\sum_{w\in V}  \mathbb{E}\bigg(\sum_{u\in V}\mathbb{I}(\mathcal{P}_{u\mapsto w})
    \bigg)
    \mathbb{E}\bigg(\sum_{v\in V}\mathbb{I}(\mathcal{P}_{v\mapsto w}) \bigg)
\end{align}
  In the second equality we used linearity of expectations and the fact that if
 $(P_1,P_2)\in \mathcal{P}_{u,v\mapsto w}$ then $P_1,P_2$ do not share 
  any vertices other than $w$. Therefore since both $P_1$,$P_2$ are both directed towards $w$ we have that they appear independently in $G_D$.
 In the inequality we used that  $\mathcal{P}_{u,v\mapsto w}\subset \mathcal{P}_{u\mapsto w}\times \mathcal{P}_{v\mapsto w}$. 
Observe that  for every $w_1,w_2\in V$ we have
\begin{align}
\mathbb{E}\bigg(\sum_{v\in V}\mathbb{I}(\mathcal{P}_{v\mapsto w_1}) \bigg)=\mathbb{E}\bigg(\sum_{v\in V}\mathbb{I}(\mathcal{P}_{v\mapsto w_2}) \bigg).
\end{align}
Therefore for $w\in V$
\begin{align}
    \mathbb{E}\bigg(\sum_{v\in V}\mathbb{I}(\mathcal{P}_{v\mapsto w}) \bigg)
    &=\frac{1}{|V|}\sum_{w'\in V}\mathbb{E}\bigg(\sum_{v\in V}\mathbb{I}(\mathcal{P}_{v\mapsto w'}) \bigg) = \frac{1}{2^{n}}\sum_{v\in V}
    \mathbb{E}\bigg(\sum_{w'\in V}\mathbb{I}(\mathcal{P}_{v\mapsto w'}) \bigg)
\\&\leq  2^{-n}\sum_{v\in V} (n^2+1) =(n^2+1).
\end{align}
 In the inequality we used that out of any vertex there are at most $n^2+1$ di-paths of length at most $n^2$ (counting the path of length 0). Substituting in \eqref{pathcount} we get
\begin{align}
    \mathbb{E}(Z_S)\leq \sum_{w\in V} (n^2+1)^2 =(n^2+1)^2 2^n.
\end{align}
Thus  the Markov inequality gives us 
\begin{align}
    \mathbb{P}\big( Z_S \geq n^6 2^n\big)=o(1).
\end{align}
Hence w.h.p.\@ $Z=Z_S+Z_L \leq n^62^n$.\vspace{3mm}\\Now let 
$C$  be the size of a largest component in $Q^n(1)$. By summing over all the unordered pairs of vertices that belong to a largest component of $Q^n(1)$, including pairs of repeated vertices,  we have  \begin{align}
    C^2\leq Z.
\end{align}
Therefore  given that $Z\leq n^62^n$ we have that $C\leq n^32^{\frac{n}{2}} =o(2^n)$ and so there is no component spanning $\Omega(2^n)$ vertices.\QEDB
\subsection{Distribution of Cycles in $Q^n(1)$}
The following lemma is the first step in proving that $Q^n(2)$ has a giant component.
\begin{lem}\label{structural}
W.h.p.\@ in $Q^n(1)$ the following hold.
\begin{enumerate}[(i)]
    \item There are $(1+o(1)){2^{n-1}}/{n}$ cycles of length two.
    \item There are  $O\big({2^{n}}/{n^{1.5}}\big)$ cycles of length  greater than two.
\end{enumerate}
\end{lem}
\begin{proof}
 Since $Q^n$ is bipartite it does not  contain any odd cycles. For $\ell \in [2^{n-1}]$ let $X_{2\ell}$ be the number of cycles of length $2\ell$ and set $X=\sum_{i\in \mathbb{N}}X_{2i}$.
Also let $Y_{2\ell}$ be the number of vertices that lie on cycles of length $2\ell$. Clearly $X_{2\ell}= \frac{1}{2\ell} Y_{2\ell} \leq\frac{1}{2} Y_{2\ell}$.
We start by proving (i). 
\vspace{3mm}
\\ For $v\in V$ we denote by $e_1(v)$ the edge that is chosen by $v$ in the generation of $Q^n(1)$. 
We first bound $\mathbb{E}(X_2)$. 
Then we use McDiarmid's inequality to establish concentration of $X_2$. 
\vspace{3mm}
\\A vertex $v$ lies on a 2-cycle in $Q^n(1)$ if $v$ and one of its neighbors choose the edge between them, thus with probability $n\cdot1/n^2$. 
Therefore $\mathbb{E}(Y_2)=\frac{2^n}{n}$.  
Observe that $Y_2$ is a function of $e(v_1),...,e(v_{2^n})$. Moreover if we alter one of the edges then we may destroy or/and we may create at most one cycle of size 2. Thus Lemma \ref{con2}, with $c_i=1$ for $i\in [2^n]$, implies 
$$\P\big(|Y_2-\mathbb{E}(Y_2)| \geq  2^{\frac{3n}{4}}\big) \leq 2\exp\bigg\{\frac{2^{\frac{3n}{2}}}{2^n}\bigg\}=o(1).$$
Hence w.h.p.\@ $X_2=Y_2/2 =(1+o(1))\frac{2^{n-1}}{n}$.
\vspace{3mm}
\\In order to prove (ii) (i.e.\@ to bound the number of cycles of size at least 4) for every $v\in V$ we define
the sequence $\{S_i(v)\}_{i\in \mathbb{N}}$  
with elements in  $[n]$ 
as follows. For $i \in \mathbb{N}$  we set $S_i(v)$ to be the coordinate in which $f^{i-1}(v)$ and $f^i(v)$ differ. Note we can deduce  $\{f^i(v)\}_{i\geq0}$ from $\{S_i(v)\}_{i\in \mathbb{N}}$. 
\begin{rem}
Given that $\big(v,f(v)\big),...,\big(f^{i-1}(v),f^i(v)\big)$ do not span a cycle of any length,  $f^{i+1}(v)$ is independent of  $v,f(v),...,f^{i}(v)$ and is distributed uniformly at random over all the neighbors of $f^i(v)$. Hence $S_{i+1}(v)$ is distributed uniformly over $[n]$. 
\end{rem}
We now reveal the terms of $\{S_i(v)\}_{i\in \mathbb{N}}$ one by one.
Let $\mathcal{E}_{2\ell}(v)$ be the event that $v$ 
 belongs to a cycle of length $2\ell$.
 If $\mathcal{E}_{2\ell}(v)$ occurs
 then $f^{2\ell}(v)=v$. Hence every element in $[n]$ appears an even number of times in the first $2\ell$ terms of the sequence $\{S_i(v)\}_{v\in \mathbb{N}}$.
That is because if $x\in [n]$ appears  an odd number of times among those terms
then $v$ and $f^{2\ell}(v)$ differ in their $x$-th entry.
Therefore in the event $\mathcal{E}_{2\ell}(v)$
 we can pair the first $2\ell$ terms of $\{S_i(v)\}_{v\in \mathbb{N}}$
   such that
 in every pair
 the two elements are the same. We can pair the $2\ell$ terms by first choosing $\ell$ terms out of the $2\ell$ and then pairing them with the remaining ones. That is in $\binom{2\ell}{\ell} \ell!$ ways. 
Assume that the $a^{th}$ term of the sequence, $S_{a}(v)$, is paired with the $b^{th}$ term $S_{b}(v)$, $a<b$. Then given that $\big(v,f(v)\big),...,\big(f^{b-2}(v),f^{b-1}(v)\big)$ do not span a cycle of any length,  $S_b(v)$ equals $S_a(v)$ with probability $\frac{1}{n}$.  Hence,
\begin{align}
  \mathbb{E}\bigg(\sum_{\ell=2}^{n/4}X_{2\ell}\bigg)&\leq  \mathbb{E}\bigg(\sum_{\ell=2}^{n/4}Y_{2\ell}\bigg) \leq 2^n \sum_{\ell=2}^{n/4} \binom{2\ell}{\ell} \ell! \bigg( \frac{1}{n} \bigg)^{\ell}
   \leq 2^n\sum_{\ell=2}^{n/4} \bigg(\frac{2\ell}{n}\bigg)^{\ell}=O\bigg(\frac{2^n}{n^2}\bigg).
\end{align}
The Markov inequality implies that w.h.p.\@ 
$\sum_{\ell=2}^{n/4}X_{2\ell}\leq \frac{2^n}{n^{1.5}}.$
\vspace{3mm}
\\A cycle of size $2\ell$ induces a path in $G_D$ of length $2\ell$. Therefore Lemma \ref{small paths} implies that there does not exists a cycle of size larger than $n^2$. To bound the remaining 
variables $Y_{2\ell}$, i.e.\@ when $\ell \in [n/4,n^{2}/2]$, we define  the sequence
$\{L_i(v)\}_{i\in \mathbb{N}}$.
For $i\in \mathbb{N}$ we set $L_i(v)=\{j\in[n]:(v)_j\neq \big(f^i(v)\big)_j\}$. By $\big(f^i(v)\big)_j$ we denote the $j$-th coordinate of $f^i(v)$, hence  $L_i(v)$ records the entries in which $v$ and $f^i(v)$ differ.
\vspace{3mm}
\\Recall that given that the edges $\big(v,f(v)\big),...,\big(f^{i-1}(v),f^{i}(v)\big)$ do not span a cycle of any length  $f_{i+1}(v)$ is chosen uniformly at random from the $n$ neighbors of $f^i(v)$. Let $j=S^{i+1}(v)$. If  $j\in L_i(v)$,
then $L_{i+1}(v)=L_i(v){ \setminus}\{j\}$.
On the other hand if $j\in[n]{ \setminus} L_i(v)$, then $L_{i+1}(v)=L_i(v)\cup \{j\}$. Thus, as $j$ is chosen uniformly at random from $[n]$ we have that 
$|L_{i+1}(v)|=|L_{i}(v)-1|$ with probability $\frac{|L_{i}(v)|}{n}$ and
$|L_{i+1}(v)|=|L_{i}(v)|$ with probability $\frac{n-|L_{i}(v)|}{n}$. 
\vspace{3mm}
\\Now let $\{L_i\}_{i\geq0}$ be  the biased random walk  {defined by},
$$L_{i} = \begin{cases} 0 &\mbox{if $i=0$},\\
L_{i-1}-1&\mbox{ with probability } \frac{L_{i-1}}{n} \hspace{10mm} \text{if } i\geq 1, \\ L_{i-1}+1 &\mbox{ with probability }\frac{n-L_{i-1}}{n} \hspace{6mm} \text{if } i\geq 1.
\end{cases}
$$
If $v=f^0(v)$ lies on a cycle of length $2\ell$ then $f^{2\ell}(v)=v$ and $L_{2\ell}(v)=\emptyset$. 
Therefore we can couple the sequence $\{|L_i(v)|\}_{i\geq0}$ with the bias random walk $\{L_i\}_{i\geq0}$ such that if $v$ lies on a cycle of length $2\ell$ then $L_{2\ell}=0.$ In order to finish the proof we use the following fact whose proof is given in the appendix.
\begin{lem}\label{bias}
Let $\{L_i\}_{i\geq0}$ be a biased random walk given by, 
$$L_{i} = \begin{cases} 0 &\mbox{if $i=0$},\\
L_{i-1}-1&\mbox{ with probability } \frac{L_{i-1}}{n} \hspace{10mm} \text{if } i\geq 1, \\ L_{i-1}+1 &\mbox{ with probability }\frac{n-L_{i-1}}{n} \hspace{6mm} \text{if } i\geq 1.
\end{cases}
$$
Then we have that
$\mathbb{P}\big( 
\exists \ell \in [n/4,n^2]: L_{2\ell}=0  \big)\leq n^{-4}.$
\end{lem}
Given Lemma \ref{bias} we have that 
\begin{align}
    \sum_{\ell =\frac{n}{4}}^{\frac{n^2}{2}}\mathbb{E}(X_{2\ell})&\leq 
    \sum_{\ell =\frac{n}{4}}^{\frac{n^2}{2}}\mathbb{E}(Y_{2\ell})=
    2^n\mathbb{P}\big(v \text{ belongs to a cycle of size }2\ell, \ell \in [n/4,n^2/3]\big)\\
    &\leq 2^n \mathbb{P}\big(\exists \ell \in [n/4,n^2]: L_{2\ell}=0\big)\leq n^{-4}2^n.
\end{align}
Finally the Markov inequality implies, 
$\mathbb{P}\bigg( \underset{\ell ={n}/{4}}{\overset{{n^2}/{2}}{\sum}}X_{2\ell}\geq \frac{2^n}{n^2}\bigg)=o(1). $
\end{proof}

\begin{cor}\label{number}
W.h.p.\@  $Q^n(1)$ consists of $(1+o(1)){2^{n-1}}/{n}$ connected components.
\end{cor}
\begin{proof}
The corollary follows directly from Lemma \ref{structural} and the fact that every component in $Q^n(1)$ contains a  unique cycle.
\end{proof}
\section{The emergence  of the Giant Component}
\subsection{The construction of $Q^n(1.5)$}
Let $V=V_0\cup V_1$ where $V_0=\{v\in V: 
{\sum}_{i\in[n]}(v)_i=0\mod 2\}$ and $V_1=V{ \setminus} V_0$. For $v\in V$ denote the edge found in $Q^n(1)$ $\big(v,f(v)\big)$ by $e_1(v)$.
Given $Q^n(1)$ we construct $Q^n(1.5)$ as follows. Every $v\in V_0$ independently chooses, uniformly at random, an edge $e_{1.5}(v)$ from those adjacent to it in $Q^n$, excluding $e_1(v)$. Let $E_{1.5}':=\{e_{1.5}(v):v\in V_0\}$ and define $Q^n(1.5)$ by 
$E_{1.5}:=E\big(Q^n(1.5)\big):=E_{1.5}'\cup  E\big(Q^n(1)\big)$ and $V\big(Q^n(1.5)\big):=V$.  
\begin{lem}\label{1.5}
W.h.p.\@ $Q^n(1.5)$ consist of at most $\frac{5\cdot2^n}{n\log_2\log_2n}$ components.
\end{lem}
\begin{rem}
The extra $1/\log_2\log_2n$ factor in the number of the components  will be the catalyst in our proof  for the size of the giant component in $Q^n(2)$.
\end{rem}
\begin{proof}
We split our proof into two parts. In the first one we show that an at most  $1/n$ fraction  of the vertices lie in a component of size at most $n^2$.  After this we show that the number of components of size less than $\log_2\log_2n$ in $Q^n(1.5)$ is small.   
\vspace{3mm}
\\{\bf{Claim 1.}} W.h.p.\@ in $Q^n(1.5)$ at most $\frac{2^{n+2}}{n}$ vertices lie on a component of size at most $n^2$.
\vspace{3mm}
\\ \emph{Proof of Claim 1}. Let $C_1,...,C_z$ be the components of $Q^n(1)$. For a given partition of the components of $Q^n(1)$ into two sets $P_1,P_2$ define $E_{P_1,P_2}(1):=\{uv\in E : u\in C_i, v\in C_j, C_i\in P_1 \text{ and } C_j\in P_2 \}$ and for $i=1,2$ $V(P_i)=\{v\in V: \exists C_j\in P_i \text{ with } v\in C_j \}$. Furthermore define the set of partitions $\mathcal{P}_1:=\big\{(P_1,P_2): |E_{P_1,P_2}(1)| \geq {2^{n-1}}  \big\}$. For $(P_1,P_2)\in \mathcal{P}_1$ let $\mathcal{B}_{1.5}(P_1,P_2)$ be the event that $E_{P_1,P_2}(1)\cap E_{1.5}'=\emptyset  $. Then 
\begin{align}\label{uselater}
    \mathbb{P}\big(\mathcal{B}_{1.5}(P_1,P_2)\big)
    &= \prod_{v\in V(P_1)\cap V_0}
    \Bigg(1-\frac{d\big(v,V(P_2)\big)}{n-1} \Bigg)
    \prod_{v\in V(P_2)\cap V_0}
    \Bigg(1-\frac{d\big(v,V(P_1)\big)}{n-1} \Bigg)\\
    &\leq \exp\Bigg\{-\sum_{v\in V(P_1)\cap V_0}
    \frac{d\big(v,V(P_2)\big)}{n}
    -\sum_{v\in V(P_2)\cap V_0}
    \frac{d\big(v,V(P_1)\big)}{n}
    \Bigg\}\\
    &= \exp\bigg\{- \frac{ | E_{P_1,P_2}{(1)}| }{n}\bigg\} \leq 
    \exp\bigg\{- \frac{ 2^{n-1}}{n}\bigg\}.
\end{align}
To go from the second to the third line we used that every edge in $E_{P_1,P_2}(1)$ has one endpoint in each of $V(P_1),V(P_2)$ and that exactly one of those belongs to $V_0$.
Corollary \ref{number} implies that  $|\mathcal{P}_1|\leq 2^{\frac{(1+o(1))2^{n-1}}{n}}$.
Therefore 
\begin{align}\label{partition1.5}
    \mathbb{P}\big ( \exists (P_1,P_2)\in \mathcal{P}_1 : \mathcal{B}_{1.5}(P_1,P_2) \text{ occurs }\big)
    \leq 2^{\frac{(1+o(1))2^{n-1}}{n}} \cdot \exp\bigg\{- \frac{ 2^{n-1}}{n}\bigg\}=o(1).
\end{align}
Now let $C_1',..,C_w'$ be the components of $Q^n(1.5)$ where $C_1',...,C_s'$, $s\leq z$, are all the components of size at most $n^2$. Furthermore set $E_B(1.5)=\{uv\in E: u\in C_i', v\in C_j' \text{ and } i\neq j\}$.
Let $\mathcal{B}_{1.5}$ be the event that 
 more than $\frac{2^{n+2}}{n}$ vertices lie on a component of size at most $n^2$. If $\mathcal{B}_{1.5}$ occurs then
 \begin{align}
     \big| E_B(1.5)\big|&= \frac{1}{2}\sum_{i\in [z]}E(C_i,V{ \setminus} C_i) 
     \geq  \frac{1}{2}\sum_{i\in [s]}E(C_i,V{ \setminus} C_i)
    \geq \frac{1}{2}\sum_{i\in [s]}(n|C_i|-|C_i|\log_2|C_i|)\\
    &\geq \frac{1}{2}\sum_{i\in [s]}
    (n-\log_2n^2)|C_i|
    \geq  \frac{[1-o(1)]n}{2}\cdot \frac{2^{n+2}}{n}=[1-o(1)]2^{n+1}.
 \end{align}
We now partition $C_1',...,C_z'$ into two sets $P_1'',P_2''$ by independently including each $C_i'$ in $P_1''$ with probability 0.5 and into $P_2''$ otherwise. 
Let $E_{P_1'',P_2''}(1.5)=\{uv\in Q^n: u\in C_i', v\in C_j', C_i'\in P_1'' \text{ and } C_j'\in P_2''  \}$. Then
\begin{align}
\mathbb{E}\big(\big| E_{P_1'',P_2''}(1.5)\big| \big)=\sum_{e\in E_B}\mathbb{P}\big(e\in E_{P_1'',P_2''}(1.5)\big)=\sum_{e\in E_B}0.5\geq [1-o(1)]2^{n}.
\end{align}  
Therefore if the event $\mathcal{B}_{1.5}$ occurs then there exists a partition $P_1',P_2'$ of $C_1',...,C_w'$  such that $E_{P_1',P_2'}(1.5)\geq [1-o(1)]2^{n}$. Each of $C_i'$ is a union sets in  $\{C_1,C_2,...,C_w\}$. In addition $C_i'$ are disjoint. Thus $P_1',P_2'$ induces a partition of $\{C_1,C_2,...,C_w\}$ into two sets $P_1,P_2$ such that $|E_{P_1,P_2}(1)|=| E_{P_1',P_2'}(1.5)|\geq [1-o(1)]2^{n}\geq 2^{n-1}$ and $E_{P_1,P_2}(1)\cap E_{1.5}'=\emptyset$. Hence, since  $\mathcal{B}_{1.5}$ implies that there exists  $(P_1,P_2)\in \mathcal{P}_{1}$ such that $\mathcal{B}_{1.5}(P_1,P_2)$  occurs, \eqref{partition1.5} implies  that $\mathbb{P}\big(\mathcal{B}_{1.5}\big)=o(1).$ \QEDB
\vspace{3mm}
\\ {\bf{Claim 2.}} W.h.p.\@ in $Q^n(1.5)$ there are at most $\frac{2^{n}}{n^{1.5}}$ vertices that lie on a connected component $C_T$  which satisfies the following: i) $|C_T|\leq s=\log_2\log_2n$, ii) any cycle spanned by $C_T$ of size larger than 2 has a nonempty intersection with $E_{1.5}'$.
\vspace{3mm}
\\ \emph{Proof of Claim 2}. Let $M$ be the number of  connected components satisfying the above  conditions and let $H$ be one of them. Then the following are true. $H$  spans  a component  $H_1$ in $G_D$ of size at most $s$. Moreover $H_1$ spans a cycle $C_{H_1}$ of size 2 (due condition ii). In addition if we let  
$v\in C_{H_1} \cap V_0$ then $\exists w\in V$ and a component  $H_2$ in $G_D$ such that $e_{1.5}(v)=(v,w)$, 
 $w\in H_2$, $|H_2|\leq s$  and $H_2$ contains a 2-cycle.
\vspace{3mm}
\\We can specify an instance of the above configuration as follows. First we specify $H_1$  
by choosing two vertices $v_1,v_2$ and two in-arborescences $T_1,T_2$ rooted at $v_1$ and $v_2$ respectively. Furthermore we request that $v_2$ is a neighbor of $v_1$ (w.l.o.g\@ $v_1\in V_0$) so that is feasible for 
$\{v_1,v_2\}$ to span a cycle in $G_D$. In addition  $T_1,T_2$ must satisfy
  $|T_1|=s_1, |T_2|=s_2$ for some  $s_1,s_2\leq s$. Thus in $G_D$, $H_1$ consists of the 2 in-arborescences  $T_1,T_2$ and the directed cycle $v_1,v_2,v_1$.
Thereafter we choose $(v_1,w)=e_{1.5}(v)$. In the case that $H_1=H_2$ we choose $w\in V(H_1)$ and we set $s_3=0$.
 Otherwise we choose $w\notin V(H_1)$ 
and we also choose a tree $T_3$ that contains $w$ such that it is vertex disjoint from $H_1$ and it has size $s_3\leq s$. 
We then choose a vertex $v_3$ in $V(T_3)$ and a neighbor of it $v_4$. Then we direct every arc on $T_3$ either towards $v_3$ or $v_4$ to create two in-arborescences $T_4$, $T_5$ rooted at $v_3$ and $v_4$ respectively. Thus in $G_D$, $H_2$ consists of the 2 in-arborescences  $T_3,T_4$ and the directed cycle $v_3,v_4,v_3$. Furthermore $w$ belongs to one of $T_3,T_4$.
 Observe that the probability of all the arcs in $E(T_1)\cup E(T_2)\cup\{(v_1,v_2),(v_2,v_1)\}$
 occurring in $G_D$ is $\frac{1}{n^{s_1+s_2}}.$ In addition $(v_1,w)$ occurs with probability $\frac{1}{n-1}$ and in the case that $H_1\neq H_2$ the arcs in $E(T_4)\cup E(T_5)\cup\{(v_3,v_4),(v_4,v_3)\}$
occur in $G_D$ with probability $\frac{1}{n^{s_3}}.$
\vspace{3mm}
\\There are $2^n$ choices for $v_1$ thereafter $n$ choices for $v_2$. Furthermore from Lemma \ref{trees1}
we have that there are at most $\binom{n}{s_i} / [(n-1)s_i+1]\leq n^{s_i}/ [(n-1)s_i+1]$ choices for $T_i$, $i\in \{1,2,3\}$. In the case that $w\in V(H_1)$ or equivalently in the case that $s_3=0$
there are at most $s_1+s_2\leq 2s$ choices for $w$. Otherwise there are at most $n$ of them. 
Therefore,
\begin{align}
    \mathbb{E}(M)&\leq \sum_{s_1,s_2\leq s}n2^n\prod_{i\in[2]} \frac{n^{s_i}}{(n-1)s_i+1}\cdot \frac{1}{n^{s_1+s_2}}  \cdot\frac{1}{n-1}
    \bigg({2s}+\sum_{s_3=1}^{s}   \frac{n^{s_3}}{(n-1)s_3+1} \cdot n\cdot   \frac{1}{n^{s_3}}\bigg) \\
    & \leq \sum_{s_1,s_2\leq s}n2^n\cdot \frac{1}{n^2(n-1)} 
    ({2s}+s ) \leq \frac{s^2\cdot 3s \cdot 2^n }{n(n-1)}=o\bigg(\frac{2^n\cdot \log_2 n}{n^2}\bigg).
\end{align}
The Markov inequality implies that $\mathbb{P}\big(M\geq \frac{2^n}{n^{1.5}}\big)=o(1)$. \QEDB
\vspace{3mm}
\\In $Q^n(1.5)$ there are  w.h.p.\@ at most $\frac{2^n}{n^2}$ components  of size at least $n^2$. 
Furthermore  Claim 1 implies that  w.h.p.\@ there are at most $\frac{2^{n+2}}{n\log_2\log_2n}$ components  of size in $[\log_2\log_2n,n^2]$. 
Thereafter Claim 2 implies that  w.h.p.\@  there are at most $\frac{2^{n}}{n^{1.5}}$ components of 
size at most $s=\log_2\log_2n$ such  that any cycle of length larger than 2 spanned by such a component has a nonempty intersection with $E_{1.5}'$.
Any component that we have not accounted for must span a cycle of size at least 4 whose edges lie in $E\big(Q^n(1)\big)$. The number of such components is bounded by the number of cycles of size larger than 4 in $Q^n(1)$ which is $O\big(\frac{2^n}{n^{1.5}}\big)$ (by Lemma \ref{structural}). 
Summing up we have that $Q^n(1.5)$  consists of at most  $\frac{5\cdot 2^{n}}{n\log_2\log_2n}$ connected components.
\end{proof}
\begin{lem}
W.h.p\@ $Q^n(2)$ has a connected component of size $[1-o(1)]2^n$.
\end{lem}
\begin{proof}
We generate $E_2'$ by  independently including  for every $v\in V_1$ an edge that is adjacent to $v$ chosen uniformly at random from the $n-1$ edges adjacent to it in $Q^n$ excluding $e_1(v)$. 
We then extend $Q^n(1.5)$ to $Q^n(2)$ by adding to its edge set  all the edges in $E_2'$.
Henceforth we follow the same argument given in the prove of Lemma \ref{1.5}${\setminus}$Claim 1.
\vspace{3mm}
\\Let $W_1,...,W_z$ be the components of $Q^n(1.5)$. For a given partition of the components of $Q^n(1.5)$ into two sets $P_1,P_2$ define $E_{P_1,P_2}(1.5)=\{uv\in E : u\in W_i, v\in W_j, W_i\in P_1 \text{ and } W_j\in P_2 \}$. Furthermore define the set of partitions $\mathcal{P}_{1.5}:=\big\{(P_1,P_2): |E_{P_1,P_2}(1.5)| \geq \frac{25\cdot 2^{n}}{\log_2\log_2 n}  \big\}$. For $(P_1,P_2)\in \mathcal{P}_{1.5}$ let $\mathcal{B}_2(P_1,P_2)$ be the event that $E_{P_1,P_2}(1.5)\cap E_{2}'=\emptyset  $. Then  similar calculations to those done for \eqref{uselater} imply 
\begin{align}
    \mathbb{P}\big(\mathcal{B}_2(P_1,P_2)\big)
    \leq  \exp\bigg(- \frac{\big| E_{P_1,P_2}(1.5)\big| }{n}\bigg) =
    \exp\bigg(- \frac{25\cdot 2^{n}}{n\log_2\log_2 n} \bigg).
\end{align}
As argued earlier $Q^n(1.5)$ w.h.p.\@ consists of at most $\frac{5\cdot 2^{n}}{n\log_2\log_2 n}$ connected components  and hence  $|\mathcal{P}_{1.5}|\leq 2^{\frac{5\cdot 2^{n}}{n\log_2\log_2 n}}$.
Therefore 
\begin{align}\label{partition2}
    \mathbb{P}\big ( \exists (P_1,P_2)\in \mathcal{P}_{1.5} : \mathcal{B}_2(P_1,P_2) \text{ occurs }\big)
    \leq 2^{\frac{5\cdot 2^{n}}{n\log_2\log_2 n}} \cdot \exp\bigg(- \frac{25\cdot 2^{n}}{n\log_2\log_2 n}\bigg)=o(1).
\end{align}
Now let $W_1',..,W_q'$ be the components of $Q^n(2)$ and let $\mathcal{B}_2$ be the event that $Q^n(2)$ has no connected component of size larger than $h=\bigg(1-\frac{100\log 2}{\log_2\log_2n}\bigg)2^n$. Furthermore define $E_B(2)=\{uv\in E: u\in W_i', v\in W_j' \text{ and } i\neq j\}$.
If $\mathcal{B}_{2}$ occurs then
 \begin{align}
    \big|  E_B(2)\big| &= \frac{1}{2}\sum_{i\in [z]}E(C_i,V{ \setminus} C_i) 
    \geq \frac{1}{2}\sum_{i\in [z]}(n|C_i|-|C_i|\log_2|C_i|)
    \geq \frac{1}{2}\sum_{i\in [z]}
    (n-\log_2h)|C_i|\\
&    \geq  -\frac{2^n}{2}\cdot \log_2\bigg\{1-\frac{100\log 2}{\log_2\log_2n}\bigg\} \geq
-\frac{2^n}{2} \frac{1}{\log 2} \cdot \frac{-100\log 2}{\log_2\log_2n}
= \frac{50 \cdot 2^{n}}{\log_2\log_2 n}.
 \end{align}
For the last inequality we used that $-\frac{x}{\log 2}\leq -\frac{\log(1+x)}{\log 2}=-\log_2(1+x)$.
We place independently and uniformly at random  $W_1',...,W_q'$ into one of the two sets $P_1'',P_2''$. Hence $P_1'',P_2''$ form a random partition of $W_1',...,W_q'$. In expectation half of the edges in $E_B(2)$
would cross this random partition (i.e.\@ their endpoints would belong to two components not found in the same set of the partition). Thus there exists a partition of $W_1',...,W_q'$ into two sets $P_1',P_2'$ where the number of edges that cross the partition $P_1',P_2'$ is at least $\big| \frac{E_B(2)}{2}\big | \geq \frac{25 \cdot 2^{n}}{\log_2\log_2 n}$.  Furthermore $P_1',P_2'$ induces a partition $P_1,P_2$ on the components of 
$Q^n(1.5)$, $W_1,...,W_z$, where the same number of edges 
cross the partition $P_1,P_2$. 
In the event $\mathcal{B}_2$ that number  is at least  $\frac{25 \cdot 2^{n}}{n\log_2\log_2 n}$.
Hence the occurrence of the event $\mathcal{B}_2$ implies that $\exists (P_1,P_2)\in \mathcal{P}_{1.5} $ such that $ \mathcal{B}_2(P_1,P_2) \text{ occurs } $. Thus \eqref{partition2} implies that $\mathbb{P}\big(\mathcal{B}_2\big)=o(1).$ 
\end{proof}

\section{Connectivity of $Q^n(k)$ - The Lower Bound}
We still have not shown that $Q^n(2)$ is disconnected. However this fact follows from the first part of Theorem \ref{third}  where we prove that  $\P\big( Q^n(k_0) \text{ is disconnected } \big)=1-o(1)$ plus the fact that $\P\big( Q^n(k) \text{ is disconnected }\big)$ is decreasing with respect to $k$. 
\vspace{3mm}  
\\In order to prove that $Q^n(k_0) \text{ is disconnected }$  we show that  $k_0$-cubes appear as connected components of $Q^n(k_0)$. That is there exists at least one $k_0$-cube $A\subset Q^n$ such that in the $k_0$-out model every vertex in $A$ chooses its neighbors in $A$ (there are exactly $k$ such neighbors). Moreover no vertex in $V{ \setminus} A$ chooses a neighbor in $A$. We use the following definition in order to describe the two  aforementioned  events.      
\begin{defn}
For $A,B\subset V$ and $k\in \mathbb{N}$ we let ${\cal{E}}_k (A\rightarrow B)$  be  the event  that in $Q^n(k)$ every vertex in ${A}$ chooses its neighbors from $B$.
\end{defn}
\begin{lem}\label{low}
Let $k_0=\lfloor \log_2 n -2\log_2\log_2 n\rfloor $. Then w.h.p.\@ $Q^n(k_0)$ is disconnected.
\end{lem}
\begin{proof}
We say that a $k_0$-cube $H$ \emph{lies} at the level  $\ell$ of  $Q^n$ if there exist
a partition of $[n]$ into 3 sets $L_H$,$F_H$ and $U_H$ of size $\ell$, $k_0$ and $n-\ell-k_0$ respectively such that 
$\forall v \in V(H)$ the entries of $v$ corresponding to the elements in $L_H$ (and $U_H$ respectively) equal to 1 (0 respectively). The entries of $v$ corresponding to elements in $F_H$ may be either 0 or 1 i.e.\@ they are free.
\vspace{3mm}
\\Let $H_1,H_2,...H_{s}$ be the $k_0$-cubes that lie at level $\frac{n}{2}$ of $Q^n$. Furthermore let $X_i$  be the indicator of the event that $H_i$ spans a connected component of $Q^n(k_0)$ and define $X$ by  $X:=\sum_{i\in[s]}X_i.$ 

In order to specify a $k_0$-cube $H$ that lies at level $\frac{n}{2}$ we can first specify $L_H$, which can be done in $\binom{n}{\frac{n}{2}}$ ways, 
and thereafter we can specify $F_H$ which can be done in $\binom{\frac{n}{2}}{k_0}$ ways. 
Thus there are  
$s=\binom{n}{\frac{n}{2}}\binom{\frac{n}{2}}{k_0}$
$k_0$-cubes that lie at level $\frac{n}{2}$. 
Let $H$ be such a $k_0$-cube. We denote by $V_H$, $\overline{V_H}$ and $N(V_H)$ the sets $V(H)$, $V{ \setminus} V(H)$ and the neighborhood of $V(H)$ found in $V{ \setminus} V(H)$ respectively. 
Observe that  $X_H=1$ if and only if the events ${\cal{E}}_{k_0 }(V_H\rightarrow V_H)$ and 
${\cal{E}}_{k_0 }\big( N(V_H)\rightarrow \overline{V_H}\big)$ occur. 
The probability that a given vertex in $V_H$ chooses its $k_0$ neighbors from $V_H$ in $Q^n(k_0)$ is $\binom{n}{k_0}^{-1}>n^{-k_0}$.
 Hence $$\mathbb{P}\big({\cal{E}}_{k_0} (V_H\rightarrow V_H) \big) \geq (n^{-k_0})^{|V_H|}=n^{-k_02^{k_0}}.$$

On the other hand if $v\in N(V_H)$ then $v$ has exactly one entry, say $q$, in $L_H\cup U_H$ such that if $q\in L_H$ then $v_q=0$ otherwise $v_q=1$. Thus $v$ has exactly one neighbor $v'\in V_H$ ($v,v'$ differ only on their $q$-th entry). Therefore since there are $|V_H|(n-k_0)$ edges coming out of $V_H$ we have that $|N(V_H)|=  |V_H|(n-k_0)=2^{k_0}(n-k_0)$ and that the probability that a given vertex in $N(V_H)$
does not choose its neighbor in $V_H$ in the $k_0$ out model  is $\binom{n-1}{k_0}/\binom{n}{k_0}=\frac{n-k_0}{n}$.
Thus 
\begin{align}
\mathbb{P}\big({\cal{E}}_{k_0} (N(V_H)\rightarrow \overline{V_H}) \big) 
= \bigg( \frac{n-k_0}{n} \bigg)^{|N(V_H)|}\geq \bigg(1- \frac{k_0}{n} \bigg)^{n2^{k_0}}
\geq   2^{-k_02^{k_0}}.
\end{align}
Observe that $k_02^{k_0}\leq (\log_2 n-2 \log_2\log_2 n)\frac{n}{{ \log_2^2n}}$. Thus, as the events ${\cal{E}}_{k_0} \big(V_H\rightarrow V_H\big)$ and ${\cal{E}}_{k_0} \big(N(V_H)\rightarrow \overline{V_H}\big)$ are independent, we have the following.
\begin{align}
    \mathbb{E}(X)&= \binom{n}{\frac{n}{2}}\binom{\frac{n}{2}}{k_0} \mathbb{P}\bigg(  {\cal{E}}_{k_0} \big(V_H\rightarrow V_H\big) \bigg) \mathbb{P} \bigg({\cal{E}}_{k_0} \big(N(V_H)\rightarrow \overline{V_H}\big)\bigg)\\
    &\geq \frac{2^{n}}{n} \cdot n^{-k_02^{k_0}}\cdot 2^{-k_02^{k_0}} \geq 2^{c_1},\
\end{align}
where 
\begin{align}
  c_1&=n-\log_2 n -(\log_2 n-2\log_2\log_2n)\cdot  \frac{n}{\log^2_2 n}
    \cdot (1+\log_2 n)   \\
    &= -\log_2 n +2\log_2\log_2n \cdot \frac{(1+\log_2 n)\cdot  n}{\log_2^2 n} -\frac{n}{\log_2 n}
       \geq \frac{n\log_2\log_2n}{\log_2n}.
\end{align}
\\
Hence $\mathbb{E}( X)\to \infty$ as $n\to \infty$.
\vspace{3mm}
\\Now let $i\in [s]$. If $V_{H_1}\cap V_{H_i}{ \neq\emptyset} $ then { if} $X_1=1$ { then} there are no edges from { $V_{H_1}\setminus V_{H_i}$} to $V_{H_i}{ \setminus} V_{H_1}$. Thus $X_i=0$ i.e. $\P(X_i=1|X_1=1)=0$. On the other hand if $V_{H_1}\cap V_{H_i}=\emptyset $ and $N(V_{H_1})\cap N(V_{H_i})=\emptyset$  then $X_1,X_i
$ are independent { i.e} $\P(X_i=1|X_1=1)=\P(X_i=1).$
Finally if $V_{H_1}\cap V_{H_i}=\emptyset $ but $N(V_{H_1})\cap N(V_{H_i})\neq \emptyset$ we have that   $\P(X_i=1|X_1=1)\leq \P(X_i=1).$ That is because given $X_1=1$ 
every vertex $v\in N(V_{H_1})\cap N(V_{H_i})$ does not choose an edge with endpoint to $V_{H_1}$. Hence they choose their $k_0$ edges from the remaining ones, which include vertices in $V_{H_i}$, any one of which has now larger probability to be chosen. In addition
in the event that $v$ chooses an edge with an endpoint in $V_{H_i}$
we have $X_i=0$. 
In all three cases for $i\in[s]{ \setminus} \{1\}$ we have that   $\P(X_i=1|X_1=1)\leq \P(X_i=1).$ Hence
\begin{align}
\mathbb{E}(X^2)&=\mathbb{E}\bigg[\bigg(\sum_{i=1}^{s}X_i\bigg)^2\bigg]= s\cdot
\bigg[\mathbb{E}(X_1^2)+ \sum_{i=2}^{s} \mathbb{E}(X_i X_1)\bigg]\\
& =s\cdot
\bigg[\mathbb{P}(X_1=1)+ \sum_{i=2}^{s} \mathbb{P}(X_i=1 \wedge X_1=1) \bigg]\\
&= s\cdot
\bigg[\mathbb{P}(X_1=1)+ \sum_{i=2}^{s} \mathbb{P}(X_i=1\vert X_1=1)\P(X_1=1)\bigg]
\\ &\leq
s\cdot \mathbb{P}(X_1=1)\bigg[1+
\sum_{i=2}^s\mathbb{P}(X_i=1)\bigg]
\leq \mathbb{E}(X)\big[1+\mathbb{E}(X)\big]
\end{align}
Therefore, since $\mathbb{E}(X)\to \infty$, we have 
\begin{align}
&\mathbb{P}(X>0)\geq \frac{\mathbb{E}(X)^2}{\mathbb{E}(X^2)}\geq
\frac{\mathbb{E}(X)^2}{ \mathbb{E}(X)\big[1+\mathbb{E}(X)\big]}
= \frac{\mathbb{E}(X)}{1+\mathbb{E}(X)} =1-o(1).
\qedhere
\end{align}
\end{proof}
\section{$k$-Connectivity of $Q^n(k)$}
The fact that the threshold for connectivity of $Q^n(k)$ is sharp follows from the second part of Theorem \ref{third} which we restate as Lemma \ref{upper}.
\begin{lem}\label{upper}
Let  $k\geq \lceil\log_2 n -2\log_2\log_2 n\rceil +1$. Then w.h.p.\@ $Q^n(k)$ is $k$-connected.
\end{lem}
Let $k_1 =\lceil\log_2 n -2\log_2\log_2 n\rceil +1$ and let $k\geq k_1$.
In order to prove that $Q^n(k)$ is $k$-connected we use the first moment method. However we are not able to show in one go that the expected number of pairs $S,L$ such that  $L=k-1$ 
 and $S$ is disconnected from $V\setminus (S\cup L)$ in $Q^n(k)$ tends to zero as $n$ tends to infinity.
That is because the upper bounds that we use on the number of such pairs  and on the probability of the corresponding events occurring are not strong enough to yield the desired result. For deriving better bounds a better understanding of the geometry of the hypercube is essential (see Remark \ref{expl}).
 
In order to circumvent this problem we generate $Q^n(k)$ in three steps. We start by generating $G_{0}$ which is distributed as a $Q^n(k-1)$.
Thereafter we extend $G_{0}$ to $G_2$, which is distributed as a $Q^n(k)$, in two phases. In the  first phase we only allow vertices found in {\ a} small set of vertices that can be easily disconnected in $G_0$ to choose their $k^{th}$ edge adjacent to them. The remaining vertices will choose their $k^{th}$ edge { in} the second phase.
We show that after the first phase every set of vertices that can be easily disconnected is of large size. 
In this calculation   for upper bounding the number of pairs $S,L$ we use Corollary \ref{trees}.
Thereafter we argue that after the second phase no such set remains.
Here, in order to bound the number of pairs $S,L$ we make the following crucial observation.
Fix $L \subset V$. Then every set $S$ for which $\P\big(S$ is not connected to $V\setminus (S \cup L)\big)\neq 0$ is a union of components of the subgraph of $G_1$ induced by $G\setminus L$.
\subsection{Generation of the random graph sequence  $G_{0}\subseteq G_{1}\subseteq G_{2}$.}
We first generate $G_{0}$. Every vertex $v\in V$  independently chooses uniformly at random a set $E_{0}(v)$, consisting of $k-1$ edges, out of the $n$ edges incident to it in $Q^n$. We then define $G_{0}$ by $V(G_{0}):=V,$ $ E(G_{0}):=\underset{ v\in V}{\bigcup}{E_{0}(v)}.$ Clearly $G_{0}$ is distributed as a $Q^n(k-1).$  The following definitions are going to be { of use in} the constructions of $G_1$ and $G_2$.
\begin{defn}
Let $G$ be a graph. We say that a set $S\subset V$ can be $(k-1)$-disconnected in $G$ if there { exists} a set $L\subset V$ with $|L|=k-1$ such that there is no edge from $S$ to $V{ \setminus}(S\cup L)$.  If such a set $L$ exists we also say that $S$ is $L$-disconnected in $G$. 
\end{defn}
\begin{defn}\label{partition}
For $i\in\{0,1,2\}$ define
$$\mathcal{S}_i:=\{S\subset V: S\text { can be $(k-1)$-disconnected in $G_i$}\}.$$
Furthermore for $L\subset V$   define
$$\mathcal{S}_i(L):=\{S \subset V: S \text{ is a minimal} \text{ $L$-disconnected set in }G_i \}.$$
\end{defn}
Now let $n_s=2^{k_1-0.1}$ 
and  $\mathcal{A}_{0}=\{v\in V: \exists L,S\subset V \text{ s.t. } |L|=k-1, S\in \mathcal{S}_0(L), |S|\leq  n_s\text { and } v\in S\}.$ 
In other words  $\mathcal{A}_{0}$
consists of all the vertices that belong to some $(k-1)$-disconnected set  whose size is  relatively small (we consider these vertices to be the active ones during the construction of $G_{1}$). Every vertex $v\in \mathcal{A}_{0}$  independently chooses uniformly at random an edge $e_k(v)$ out of the $n-(k-1)$ edges that are incident to it in $Q^n$ and do not belong to $E_{0}(v)$. We let the set of newly chosen edges be $E_{1}'$ and we define $G_{1}$ by $V(G_{1}):=V$, $E(G_{1}):=E(G_{0})\cup E_{1}'$. 
\vspace{3mm}
\\ We finally extend $G_{1}$ to $G_{2}$. We let  $\mathcal{A}_{1}=V{ \setminus} \mathcal{A}_{0}$. Every vertex in $\mathcal{A}_{1}$ independently  chooses uniformly at random an edge $e_k(v)$ out of the $n-(k-1)$ edges that are incident to it in $Q^n$ and do not belong to $E_{0}(v)$. We let the set of newly chosen edges be $E_{2}'$. 
Finally we define $G_{2}$ by $V(G_{2}):=V$, $E(G_{2}):=E(G_{1})\cup E_{2}'$. Observe that once we construct $G_{2}$ we have that every vertex $v$ has chosen a set of exactly $k$ edges uniformly at random from all the edges incident it. Since $E(G_{2})=\underset{v\in V}{\bigcup}E_2(v)$ we have that $G_{2}$ is distributed as a $Q^n(k)$. 
The following remarks can be made concerning definition \ref{partition}.
\begin{rem}
For $i\in \{0,1,2\}$, every set in $\mathcal{S}_i(L)$ is connected in $G_i$ hence in $Q^n$.  Therefore its vertices span a connected subgraph of $Q^n$.
\end{rem}
\begin{rem}
For $i\in \{0,1,2\}$, every $L$-disconnected set in $G_i$ is the union of sets in $\mathcal{S}_i(L)$ hence,   $$\min\{|S|: S \in \mathcal{S}_i\}=\underset{L \in \binom{V}{\ell}}{\min}\big\{\min\{|S|:S\in \mathcal{S}_i(L)\}\big\}.$$
In addition if $S\in S_i(L)$ then for $0\leq j\leq i$, since $G_j\subset G_i$, we  have that $S$ is $L$-disconnected in $G_j$ hence it is a union of sets in $\mathcal{S}_j(L)$.
\end{rem}
\begin{lem}\label{notsmall}
 W.h.p.\@ $\nexists S,L\subset V$ s.t.\@ $|L|=k-1$,  $S\in \mathcal{S}_{1}(L)$ and $|S|<n_s$.
\end{lem}
\begin{proof} 
Assume the claim is false and let $L$, $S$ be  a contradicting pair. Since $S$ is $L$-disconnected in $G_{1}$ it is also $L$-disconnected in $G_{0}$. Hence every vertex in $S$ belongs to some minimal $(k-1)$-disconnected set in $G_{0}$ of size at most $n_s$. Due to the construction of $G_{1}$ every such vertex has degree at least $k$ in $G_{1}$. Thus $S\cup L$ span at least $\frac{k|S|}{2}$ edges in $G_{1}$. In addition, by Lemma \ref{iso}, every set of $|S|+(k-1)$ vertices span at most 
$\frac{|S|+(k-1)}{2}\log_2(|S|+k-1)
\leq \frac{|S|+k}{2}\log_2(|S|+k)$ edges in $Q^n$.
Thus, since $G_1\subset Q_n$, we have
\begin{align}\label{small ccase}
\frac{k|S|}{2} \leq \frac{|S|+k}{2}\log_2(|S|+k).
\end{align}
Consequently one of the following two inequalities holds. Either
\begin{align}\label{case smalll}
\frac{k|S|}{4} \leq \frac{k}{2}\log_2(|S|+k) 
\hspace{5mm}\text{which implies }\hspace{5mm} |S| \leq 2\log_2(|S|+k)
\end{align}
or
\begin{align}\label{case small22}
\frac{k|S|}{4} \leq \frac{|S|}{2}\log_2(|S|+k) 
\hspace{5mm}\text{ which implies}\hspace{5mm} k \leq 2 \log_2(|S|+k).
\end{align}
If \eqref{case smalll} holds then, since $k$ is larger than $0.5\log_2n$, it must be the case that $|S|=o(k)$, in particular $|S|\leq \frac{k}{4}$. $S$ is $L$-disconnected in $G_{1}$ therefore every vertex in $S$ is adjacent, in $G_{1}$, to $k$ vertices in $S\cup L$. Therefore, since  $|S|\leq \frac{k}{4}$
, every vertex in $S$ is adjacent to at least $\frac{3k}{4}$ vertices in $L$. Hence any 2 distinct vertices in $S$ share at least 3 neighbors in $L$ which contradicts the fact that any two vertices in $Q^n$
have at most two common neighbors (note $S$ is $(k-1)$-disconnected and every vertex in $S$ has degree at least $k$ hence $|S|>1$). So it must be the case that \eqref{case small22} holds.   \eqref{case small22} implies that $|S|\geq k$ which implies 
\begin{align}\label{violation}
k\leq 2\log_2(2|S|)\leq 3\log_2|S|.
\end{align}
\eqref{small ccase} can be rewritten as 
\begin{align}\label{later}
\frac{k|S|}{2} \leq \frac{|S|+k}{2}\log_2(|S|+k) = \frac{|S|+k}{2}\log_2\bigg(1+\frac{k}{|S|}\bigg)+\frac{|S|}{2}\log_2|S| +\frac{k}{2}\log_2|S|.
\end{align}
Dividing throughout by $\frac{|S|}{2}$, setting $u=3\log_2|S|$ and using that  $k\leq u$
we get that in the case that the statement of our lemma is false the  following inequality holds
\begin{align}\label{laterr2}
k \leq \bigg(1+\frac{u}{|S|}\bigg)\log_2\bigg(1+\frac{u}{|S|}\bigg)
+\log_2|S|+\frac{u}{|S|}\log_2|S|=\log_2|S|+o(1).
\end{align}
Therefore a crude lower bound on $|S|$ is $|S|\geq 2^{0.9k}\geq  n^{0.8}$. 
When $|S|\geq n^{0.8}$ we have that 
 $\big(1+\frac{u}{|S|}\big)\log_2\big(1+\frac{u}{|S|}\big)
+\frac{u}{|S|}\log_2|S|
\leq \big(1+\frac{u}{|S|}\big)\cdot \frac{1}{\log 2} \cdot\frac{u}{|S|}
+\frac{u}{|S|}\log_2|S|=o({ n^{-0.6}}).$ Therefore we can replace the $o(1)$ term in \eqref{laterr2} by $o(n^{-0.6})$.
Thus as $|S|\leq n_s$ \eqref{laterr2} implies that   
$$k \leq \log_2|S|+o(n^{-6})\leq \log_2 n_s+n^{-0.6}=k_1-0.1+n^{-0.6}<k,$$
which give us a contradiction. 
\end{proof}
\begin{rem}\label{sizeproblem}
Let $k\geq k_1$. If $S,L\subset V$ are such that $|L|=k-1$ and $S\in \mathcal{S}_{1}(L)$ then { the} same argument we { used} to derive {\eqref{violation}  implies} that $k\leq  3\log_2|S|$. 
\end{rem}
\begin{lem}\label{notmedium}
Let ${ n_1=2^{\frac{n}{5}}}$. Then w.h.p.\@ $\nexists S,L\subset V$ s.t.\@ $|L|=k-1$,  $S\in \mathcal{S}_{1}(L)$ 
and $|S|\in [n_s,n_1].$
\end{lem}
\begin{proof}
In proving the above statement we observe that for every $S,L\subset V$ such that $S\in \mathcal{S}_1(L)$ and $|L|=k-1$ then we have the following. 
There exists some $L'\subset L$ such that the induced subgraph of $Q^n$ on  $S\cup L'$ is connected
and in $G_1$ every vertex in $S$ is adjacent to vertices only in $S\cup L'$.
For $s\in \mathbb{N}$ let $\mathcal{D}_s=\{(S,L'): S,L'\subset V, |L'|\leq k-1, |S|=s \text{ and } S\cup L' \text{ is connected in }Q^n\}$. 
Corollary \ref{trees} implies that 
for fixed $v\in V$  there are at most $\underset{j\in[k-1]}{\sum}(en)^{s+j}$ choices for $S\cup L'$ such that $v\in S\cup L'$ and $s+1\leq |S\cup L|\leq s+(k-1)$. Thereafter there are at most $\binom{s+k-1}{k-1}\leq (2s)^k$ ways to choose $L'$ out of $S\cup L'$.
 Thus when $k\leq 3\log_2 s$ and $s\leq n^2$ we have that  
\begin{align}\label{basesmall}
|\mathcal{D}_s|\leq \sum_{j\in[k-1]}2^n(en)^{s+j}(2s)^k&\leq k2^n(en)^{s+k}(n^3)^{k}\leq 2^n(en)^{s+5k} \leq 2^{n+2s\log_2n}.
\end{align}
On the other hand when $k\leq 3\log_2 s$ and $n^2<s$ we have 
\begin{align}\label{baselarge}
|\mathcal{D}_s|\leq \sum_{j\in[k-1]}2^n(en)^{s+j}(2s)^k&\leq k2^n(en)^{s+k}2^{2k\log_2s}\leq k2^n(en)^{s+k}2^{s} \leq 2^{n+2s\log_2n}.
\end{align}
At the same time Lemma \ref{iso} implies that  for every $(S,L')\in \mathcal{D}_s$ we have 
\begin{align}\label{armean}
\frac{1}{s}\sum_{v\in S}d(v,S\cup L')&\leq\frac{1}{s}\sum_{v\in S\cup L'}d(v,S\cup L') 
\leq\frac{1}{s}\cdot(s+k) \log_2(s+k)=(1+o(1))\log_2s. 
\end{align}
Therefore by the arithmetic-geometric mean inequality we get 
\begin{align}\label{geo}
\big \{(1+o(1))\log_2s\big \}^s\geq \bigg( \frac{1}{s}\sum_{v\in S}d(v,S\cup L')\bigg)^s \geq \prod_{v\in S} d(v,S\cup L').
\end{align}
\eqref{geo} implies 
\begin{align}\label{appl}
\prod_{v\in S}
    \Bigg(\frac{\binom{d(v,S\cup L')}{k-1}}{\binom{n}{k-1}}\Bigg)\leq
   \prod_{v\in S}\Bigg(\frac{d\big(v,S\cup L' \big)}{n}\Bigg)^{k-1}   
\leq    \bigg(\frac{(1+o(1))\log_2 s}{n}\bigg)^{s(k-1)} 
\end{align}
Hence, using that $k\leq 3\log_2s$ (see Remark \ref{sizeproblem}), we have
\begin{align}
\P\big(\exists & L,S\subset V: |L|=k-1, S\in \mathcal{S}_{1}(L) \text{ and }
|S|=s \in[n_s, n^2)\big)  \\
&\leq \P\big(\exists s \in[n_s, n^2) \text{ and } \exists S,L'\subset V :(S,L')\in \mathcal{D}_s \text{ and } E_0(v)\subset (S\cup L')\times \{v\}\big)\\
    &\leq \sum_{s=n_s}^{n^2} \sum_{(S,L')\in \mathcal{D}_s}
    \Bigg(\frac{\binom{d(v,S\cup L')}{k-1}}{\binom{n}{k-1}}\Bigg) 
      \leq\sum_{s=n_s}^{n^2} 2^{n+2s\log_2 n} \bigg(\frac{(1+o(1))\log_2 s}{n}\bigg)^{s(k-1)}  \\
    &\leq\sum_{s=n_s}^{n^2} 2^{n+2s\log_2 n} \bigg(\frac{2\log_2 n^2}{n}\bigg)^{s(k_1-1)}
    \leq \sum_{s=n_s}^{n^2} 2^{c_2(s)}=o(1).
\end{align}
{In} the second line we used  \eqref{basesmall} and  \eqref{appl}. Furthermore { in} the last line we used that  for $s\in [n_s,n^2],$
\begin{align}
    c_2(s)&=n+2s\log_2 n -\big[\log_2 n-\log_2( 4\log_2 n) \big]s (k_1-1)
    \\
    &\leq n+2n_s  \log_2 n -\big[\log_2 n-\log_2( 4\log_2 n) \big]n_s (k_1-1)\\
    &\leq n+2\cdot \frac{4n}{\log_2^2n}  \log_2 n -\big[1-o(1)\big]\cdot \log_2 n       \cdot \frac{2^{0.9} n}{\log_2^2n} \log_2 n    \\
    &= -(2^{0.9}-1)n+o(n).
\end{align}
{In} the third line we used that $ \frac{2^{0.9}n}{\log_2^2n} \leq n_s \leq \frac{4n}{\log_2^2n}.$
Similarly, we have 
\begin{align}
\P\big(&\exists L,S\subset V: |L|=k, S\in \mathcal{S}_{1}(L) \text{ and }
|S|=s \in[n^2,n_1)\big)  \\ 
&\leq\sum_{s=n^2}^{n_1} 2^{n+2s\log_2 n}
\bigg(\frac{(1+o(1))\log_2 s}{n}\bigg)^{s(k-1)}  
\\
   &\leq\sum_{s=n^2}^{n_1} 2^{n+2s\log_2 n} \bigg(\frac{(1+o(1))\log_2 n_1}{n}\bigg)^{s(k_1-1)}
    =\sum_{s=n^2}^{n_1} 2^{n+2s\log_2 n} \bigg(\frac{(1+o(1))}{5}\bigg)^{s(k_1-1)}    
    \\
    &\leq \sum_{s=n^2}^{n_1} 2^{c_3(s)} \leq \sum_{s=n^2}^{n_1}2^{ -(\log_2 5-2)n^2}=o(1).
\end{align}
{ In} the last line we used that for $s\in [n^2,n_1]$ 
\begin{align}
c_3(s)&= {n+2s\log_2 n} -(1+o(1))\log_2 5 \cdot  s(k_1-1)\\
& \leq -(1+o(1))(\log_2 5-2)s\log_2n  \leq -(\log_2 5-2)n^2.\qedhere
\end{align}
\end{proof}
\begin{rem}\label{expl}
We can  extend the above calculations to pairs of sets $S,L$ where $|S|$ satisfies 
$\frac{\log_2 |S|}{n}-2=o\big( \frac{1}{n^2}\big).$
In order to implement similar calculations for larger sets $S$ we would have to sharpen the bounds derived in \eqref{baselarge},\eqref{geo} and \eqref{appl}. Observe that \eqref{baselarge} counts each  set multiple times. Moreover we expect that as  the size of the sets that we consider in \eqref{baselarge} is increased the proportion of the 
upper bound derived over the true value is also increased.
 At the same time observe that for fixed $S,L$ in order to bound the probability that $S$ is $L$-disconnected we do not use any information about the sets $S,L$ other than their sizes. 
 On the other hand \eqref{geo}, \eqref{appl} indicate that we can relate this probability with the quantity 
 $|E\big(S,V\setminus (S\cup L)\big)|$. 
\end{rem}
{In} the proof of Lemma \ref{upper} we are also going to use the following lemma.
\begin{lem}\label{activek-1}
W.h.p\@ $|\mathcal{A}_{0}|\leq 2^{\frac{n}{10}}.$
\end{lem}
\begin{proof}
Let $ S,L \subset V$ be such that $|L|< k$, $S\in \mathcal{S}_{0}(L)$ and $|S|\leq n_s$. The same arguments used to derive \eqref{small ccase} imply 
\begin{align}\label{base1}
\frac{(k-1)|S|}{2} &\leq \frac{|S|+(k-1)}{2}\log_2(|S|+(k-1))\\
&\leq \frac{|S|+k}{2}\log_2(|S|+k)
\end{align}
Let $a=2^{k-1}-|S|{ \geq 0}$. 
By dividing throughout by $\frac{|S|}{2}$ and then substituting $|S|=2^{k-1}-a$ we have
\begin{align}
    k-1&\leq\bigg(1+\frac{k}{|S|}\bigg)\log_2(|S|+k)
    =\bigg(1+\frac{k}{2^{k-1}-a}\bigg)\log_2\bigg[2^{k-1}\bigg(1+\frac{k-a}{2^{k-1}}\bigg)\bigg] \\
   &\leq     (k-1)\bigg(1+\frac{k}{2^{k-1}-a}\bigg)+2\log_2\bigg(1+\frac{k-a}{2^{k-1}}\bigg)\\
&\leq     k-1+\frac{(k-1)k}{2^{k-1}-a}+4\frac{k-a}{2^{k-1}}.\label{zx}
\end{align}
{ In} the second line at the calculations above we used that $k=o(|S|)=o(2^{k-1}-a)$ (see Remark \ref{sizeproblem}).
 Furthermore
{in} the last inequality we used that $\forall x \in \mathbb{R}$ we have that $\log(1+x)\leq x.$ Therefore $2x >\frac{x}{\log2}\geq \frac{\log(1+x)}{\log_2}=\log_2(1+x)$.
 \eqref{zx} implies
\begin{align}
0\leq
 \frac{2^{k-1}k(k-1)}{2^{k-1}(2^{k-1}-a)} +    \frac{\big(2^{k-1}-a\big)4(k-a)}{2^{k-1}(2^{k-1}-a)}  
   \leq\frac{2^{k-1}(k^2+4k-4a)}{2^{k-1}\big(2^{k-1}-a\big)}
\end{align}
Hence $ 0\leq  k^2+4k-4a $  which implies that $a\leq 2k^2$. Thus, since
 $|S|=2^{k-1}-a$, we have $|S|\in [2^{k-1}-2k^2,n_s]\subseteq [2^{k_1-1}-2k_1^2,n_s].$ Let $n_\ell=2^{k_1-1}-2k_1^2$. 
Using, { in} the second line of the calculations below, \eqref{appl}, \eqref{basesmall} and that $(1+o(1))\log_2s \leq \log_2^2n$ for $s\leq n_s$  we have
\begin{align}
    \mathbb{E}(|\mathcal{A}_{0}|)& 
    \leq n_s\mathbb{E}\big({  |\{S: \exists L\subset V, |L|= k-1, S\in \mathcal{S}_{0}(L) \text{ and }  |S|\in [n_\ell,n_s]\big\}|})\\
   &\leq n_s \sum_{s=n_\ell}^{n_s} \sum_{(S,L')\in \mathcal{D}_s}\prod_{v\in S}
    \Bigg(\frac{\binom{d(v,S\cup L)}{k-1}}{\binom{n}{k-1}}\Bigg) 
   \leq  n_s \sum_{s=n_\ell}^{n_s} 2^{n+2s\log_2 n} \bigg(\frac{\log_2^2 n}{n}\bigg)^{s(k_1-1)}    \\
    &\leq n_s \sum_{s=n_\ell}^{n_s} 2^{c_4(s)}\leq n_s\sum_{s=n_\ell}^{n_s} 2^{\frac{n}{12}}\leq 2^{\frac{n}{11}}.
\end{align}
{ In} the last line we used that for $s\in[n_\ell,n_s]$
\begin{align}
c_4(s)&=n+2s\log_2 n -(\log_2 n-2\log_2 \log_2 n )s(k_1-1)\\
&\leq n+n_\ell \big[ 2\log_2 n -(\log_2 n-2\log_2 \log_2 n ) (k_1-1)\big]\\
      &\leq  n+\bigg(\frac{n}{\log_2^2n}-2\log_2^2n \bigg) 
 \big(2\log_2n  -{ \log_2^2 n} +4 \log_2n \cdot \log_2\log_2 n \big)=o(n).
\end{align}
{ In} the last line we used that $\frac{n}{\log_2^2n}-2\log_2^2n \leq 2^{k_0}-2k_0^2 \leq n_\ell.$
Hence by Markov's inequality we have that $\mathbb{P}\big(|\mathcal{A}_{0}|\geq 2^{\frac{n}{10}}\big) 
\leq \mathbb{E}(|\mathcal{A}_{0}|) 2^{-\frac{n}{10}}=o(1)$.
\end{proof}
{ In} the proof of Lemma \ref{upper} we are going to use the following definition. 
\begin{defn}
For $L\subset V$, $|L|=k-1$ and $2\leq j \leq 9$, with $u_j=  \min\big\{2^{\frac{(j+1)n}{10}},2^{n-1}\big\} $, we define the sets  $\mathcal{U}^j_{1}(L)$ as follows,
$$\mathcal{U}^j_{1}(L):=\bigg\{T\subset V: T\text{ is a union of sets in } \mathcal{S}_{1}(L) \text{ and } |T|\in\big[2^{\frac{jn}{10}},u_j \big]
\bigg\}.$$
\end{defn}
Observe that Lemmas \ref{notsmall}, \ref{notmedium} imply that 
$\mathcal{S}_{1}(L)$ consists of disjoint sets of size at least $2^{\frac{n}{5}}$. 
Thus $|\mathcal{S}_{1}(L)|\leq { 2^{\frac{4n}{5}}}$.
Furthermore if $T\in \mathcal{U}^j_{1}(L)$, then $T$ is  the union of at most $2^{\frac{(j-1)n}{10}}$ sets in $\mathcal{S}_{1}(L)$. 
Therefore  
\begin{align}\label{sizeu}
|\mathcal{U}^j_{1}(L)|&\leq 
\sum_{h=1}^{2^{\frac{(j-1)n}{10}}} \binom{ |S|  }{h}
\leq \sum_{h=1}^{2^{\frac{(j-1)n}{10}}} \binom{ 2^n }{h}
\leq \sum_{h=1}^{2^{\frac{(j-1)n}{10}}}\frac{2^{nh}}{{ h!}}\leq 2^{n2^{\frac{(j-1)n}{10}}}.
\end{align}
\\{\bf{Proof of Lemma \ref{upper}.}}
In the event that $G_2$ is not $k$-connected { there is} 
a set $L$ consisting of $k-1$ vertices whose removal { partitions} the rest of the vertices into connected components. The smallest one of those components is of size at most $2^{n-1}$. Therefore
we have that there exists $T,L\subset V$ and $2\leq j\leq 9$ such that $T\in \mathcal{U}^j_{1}(L)$
 and no edge in $E\big(T,V{ \setminus} (T\cup L)\big)$ appears in $E_2'$ i.e.\@ 
 $\big\{e_k(v):v\in T\cap \mathcal{A}_{1} \big\}\cap E\big(T,V{ \setminus} (T\cup L)\big) =\emptyset.$
For a fixed such triple $T,L,j$ let $\mathcal{B}(T,L,j)$ be the event that 
$\big\{e_k(v):v\in T\cap \mathcal{A}_{1} \big\}\cap E\big(T,V{ \setminus} (T\cup L)\big) =\emptyset $. Then
\begin{align}
\mathbb{P}\big( \mathcal{B}(T,L,j) \big)&
\leq \prod_{v\in T\cap \mathcal{A}_{1}}\mathbb{P}\big[ e_k(v) \notin E\big(v,V{ \setminus} T\cup L\big) \big]  
\leq \prod_{v\in T\cap \mathcal{A}_{1}} \bigg[ 1-\frac{d(v,V{ \setminus} T\cup L)}{n-(k-1)} \bigg] \\
&\leq   \prod_{v\in T\cap \mathcal{A}_1} \exp\bigg\{-\frac{d(v,V{ \setminus} T\cup L)}{n-(k-1)}\bigg\}\leq \exp\bigg\{-\frac{1}{n}\underset{v\in T\cap \mathcal{A}_1}{\sum}d(v,V{ \setminus} T\cup L)\Bigg\} \\
&\leq \exp\bigg\{-\frac{1}{n}\bigg(\underset{v\in T}{\sum}d(v,V{ \setminus} T) -n|L|-n|T { \setminus} \mathcal{A}_{1}| \bigg)\Bigg\} 
 \\
&\leq 
\exp\bigg\{-\frac{1}{n} \big(n|T|-|T|\log|T| \big)+ k +2^{\frac{n}{10}}\Bigg\}\\
&\leq \exp\bigg\{-\frac{ 2^{\frac{jn}{10}}}{n} \bigg(n-\log_2 2^{\frac{jn}{10}}\bigg) + k +2^{\frac{n}{10}} \Bigg\}
\leq \exp\bigg\{-\frac{2^{\frac{jn}{10}}}{20n}\bigg\}.
\end{align}
To go from the third to the fourth line we used Lemma \ref{iso} and  that $ |T{ \setminus} \mathcal{A}_{1}|=|T\cap \mathcal{A}_{0}|\leq  2^{\frac{n}{10}}$ (see Lemma \ref{activek-1}).
Thereafter we used Remark \ref{iso2}. 
{ In} the last inequality we  used that $2\leq j \leq 9$. Finally we have,
\begin{align}
\mathbb{P}&\big(G_2\text{ is not k-connected}\big)= 
  \P\big( \exists L,T\subset V\text{ and } 2\leq j\leq 9: \mathcal{B}(T,L,j) \text{ occurs}\big)\\
 &\leq\sum_{j=2}^9 \sum_{L\in \binom{V}{k-1}}\sum_{T\in \mathcal{U}_{1}^{j}(L)}   
  \exp\bigg\{-\frac{2^{\frac{jn}{10}}}{20n}\bigg\}
\leq\sum_{j=2}^9 \sum_{L\in \binom{V}{k-1}} 2^{n2^{\frac{(j-1)n}{10}}}  \exp\bigg\{-\frac{2^{\frac{jn}{10}}}{20n}\bigg\}=o(1). 
\end{align}
Since $G_2$ is distributed has the same distribution with $Q^n(k)$ the statement of Lemma \ref{upper} 
follows. \QEDB
\vspace{3mm}
\\A question that now arises is the following. For $k\geq k_1$ can $Q^n(k)$ be $\ell$-connected, for some $\ell>k$? We answer this question negatively { in our next lemma} with which we close this section.
\begin{lem}
Let $1\leq k\leq n-1$. Then w.h.p.\@ $Q^n(k)$ contains a vertex of degree $k$ { and} hence $Q^n(k)$ is not $(k+1)$-connected.
\end{lem}
\begin{proof}
Let $v\in V$. In $Q^n(k)$, $v$ has degree $k$ in the event that the $(n-k)$ edges not selected by $v$  do not belong to $Q^n(k)$. That is those $(n-k)$ edges are not selected by their other endpoint either.  Each of those edges is selected by their other endpoint independently with probability $\frac{k}{n}$. 
Hence
\begin{align}
p_k=\mathbb{P}( v \text{ has degree $k$ in } Q^n(k)
=\bigg(1-\frac{k}{n}\bigg)^{n-k}=\Bigg[\bigg(1-\frac{k}{n}\bigg)^{1-\frac{k}{n}}\Bigg]^n      .
\end{align}
$p_k$ is minimized when $1-\frac{k}{n}=e^{-1}$ thus $p_k\geq e^{-\frac{n}{e}}\geq 0.6^n.$ The degrees of any set of vertices which are at distance at least three from each other are independent. We can greedily select such a set $S$, of size at least $\frac{2^n}{n^2}$, by sequentially  including a { non-deleted} vertex and then deleting all the vertices at distance at most 2 from it. For $v\in V$ let $d_k(v)$ be the degree of $v$ in $Q^n(k)$.
Therefore
\begin{align}
\P\big[Q^n(k) \text{ is $(k+1)$ connected}\big]& \leq \mathbb{P}\big(\nexists v\in S:  d_k(v)=k 
\big)=(1-p_k)^{|S|} \leq (1-p_k)^{\frac{2^n}{n^2}} \\
&\leq e^{- \frac{p_k2^n}{n^2}}\leq e^{- \frac{0.6^n\cdot 2^n}{n^2}}=o(1). \qedhere
\end{align}
\end{proof}

\section{Final Remarks}
In this paper we { have} established the connectivity threshold for the random { subgraph} of the { $n$-cube that is} generated by the $k$-out model. When $k$ is below the threshold { $k_1$} the giant components consists of all but $o(2^n)$ vertices. Furthermore a calculation similar to the one given at the proof of Lemma \ref{notmedium} give us that  when $k$ is below { this} threshold $Q^n(k)$ does not have any components of size in $\big[\frac{2n}{\log_2 n},2^{\frac{n}{5}}\big].$ Hence it would be interesting to investigate { the size of the second largest component}. 

On the other hand, when $k$ is at least { $k_1$} we showed $Q^n(k)$ is far more than just connected, it is $k$-connected. In the proof of the $k$-connectivity we used the following properties of $Q^n(k)$. Let $N=2^n$ then $Q^n$ is a graph on $N$ vertices of maximum degree $\log_2 N$ such that for any partition $S,V{ \setminus} S$ there are at least $|S|(\log_2N-\log_2|S|)$ edges crossing the partition. In addition any two vertices have at most $0.25\log_2 \log_2 N$ common neighbors.
 Therefore by repeating the arguments given in this paper we have the following. Every random subgraph of a graph { on $N$ vertices} that { satisfies} the aforementioned properties { and is} generated by the $k$-out model, where $k\geq k_1$, is $k$-connected.   
An interesting question would therefore be to state more general conditions of { a} similar flavor, such that the random subgraph of a graph that { satisfies these} conditions { and is} generated by the $k$-out model
is $k$-connected (or even { just} connected).

\bibliographystyle{abbrv}
\bibliography{K_ot}

\appendix
\section{Proof of Lemma \ref{bias}}
\begin{proof}
We start by proving that 
\begin{align}\label{appeq1}
\mathbb{P}\bigg(L_\frac{n}{5}\leq \frac{n}{20}\bigg)\leq e^{-10^{-3}n}
\end{align}
Observe that for $i\leq \frac{n}{5}$, $L_i\leq \frac{n}{5}$ hence $\mathbb{P}\big(L_{i+1}=L_i+1\big)\geq 0.8$. In the event that $L_{\frac{n}{5}}\leq \frac{n}{20}$
we have that $|\{i\leq \frac{n}{5}:L_{i+1}\neq L_i+1\}|\geq 0.5 \big(\frac{n}{5}-\frac{n}{20}\big)=\frac{3n}{40}$. Equivalently we have   $|\{i\leq \frac{n}{5}:L_{i+1}= L_i+1\}|\leq \frac{n}{5}-\frac{3n}{40}=\frac{n}{8}.$
Therefore,
\begin{align}
    \mathbb{P}\bigg(L_{\frac{n}{5}}\leq \frac{n}{20}\bigg)&\leq \mathbb{P}\bigg[Bin\bigg(\frac{n}{5},0.8\bigg)\leq \frac{n}{8}\bigg]
    \leq \exp \bigg[ \bigg(\frac{\frac{4}{25}-\frac{1}{8}}{\frac{4}{25}}\bigg)^2\frac{\frac{4}{25}n}{3} \bigg]\leq e^{-10^{-3}n}.
\end{align}
{ In} the second inequality we used Lemma \ref{con1}.
Our second step is to show that for $i\in[n^2]$
\begin{align}\label{appeq2}
\mathbb{P}\bigg(L_{i+\frac{n}{40}}\leq \frac{n}{20}\bigg\vert L_i \geq \frac{n}{20}\bigg)\leq e^{-10^{-3}n}.
\end{align}
Observe that $|L_{i+\frac{n}{40}}- L_i| \leq \frac{n}{40}$. Therefore 
\begin{align*}
    \mathbb{P}\bigg(L_{i+\frac{n}{40}}\leq \frac{n}{20}\bigg\vert L_i > \frac{n}{20}+\frac{n}{40}\bigg)=0.
\end{align*}
On the other hand if $L_i\leq \frac{n}{20}+\frac{n}{40}$ we have that $L_j\leq \frac{n}{20}+\frac{n}{40}+\frac{1n}{40}=\frac{n}{10} $ for every $j\in \big[i,i+\frac{n}{40} \big]  $ hence $\mathbb{P}\big(L_{j+1}=L_j+1\big)\geq \frac{9}{10}$. In the event that $L_{i+\frac{n}{40}}\leq \frac{n}{20}$
we have 
$$\bigg|\bigg\{j\in \bigg[i,i+\frac{n}{40} \bigg] :L_{j+1}\neq L_j+1\bigg\}\bigg|\geq 0.5\cdot \frac{n}{40} $$ 
or equivalently 
$\big|\big\{j\in \big[i,i+\frac{n}{40} \big] :L_{j+1}= L_i+1\big\}\big |\leq \frac{n}{80}.$
Therefore,
\begin{align}
\mathbb{P}\bigg(L_{i+\frac{n}{40}}\leq \frac{n}{20}\bigg\vert L_i \geq \frac{n}{20}\bigg)&\leq
\mathbb{P}\bigg(L_{i+\frac{n}{40}}\leq \frac{n}{20}\bigg\vert \frac{n}{20}\leq L_i \leq \frac{n}{20}+\frac{n}{40}\bigg)
\\&\leq \mathbb{P}\bigg[Bin\bigg(\frac{n}{40},\frac{9}{10} \bigg)\leq \frac{n}{80}\bigg]\leq \exp \bigg[ \bigg(\frac{\frac{9}{400}-\frac{1}{80}}{\frac{9}{400}}\bigg)^2\frac{\frac{9}{400}n}{3} \bigg]\leq e^{-10^{-3}n}.
\end{align}
{ In} the third inequality we once again  used Lemma \ref{con1}.
\vspace{3mm}
In the event that $L_{2i}=0$ for some $i \in \big[\frac{n}{4},n^2\big]$  we have that either $L_{\frac{n}{5}} \leq \frac{n}{20}$ or there exist $\frac{n}{5} \leq i<n^2$ such that $L_{2i-\frac{n}{40}}\geq \frac{n}{20}$ but $L_{2i}\leq \frac {n}{20}$. Hence
\begin{align}
    \mathbb{P}\big(\exists &i\in [n/4,n^2]:  {L_{2i}}=0 \big) \leq n^2\cdot e^{-10^{-3}n}=o(n^{-4}). 
    \qedhere
\end{align}
\end{proof}

\end{document}